\documentclass[11pt]{article}
\usepackage{amsfonts}
\usepackage{color}
\usepackage{amsmath,amssymb,comment}
\newtheorem{theo}{Theorem}[section]
\newtheorem{lem}[theo]{Lemma}

\newtheorem{defi}[theo]{Definition}

\newcommand{\mysection}[1]{\section{#1} \setcounter{equation}{0}}
\newcommand{\proof}{{\sc Proof.} \quad}
\newcommand{\proofc}{{\sc Proof} \ }
\newcommand{\be}{\begin{equation} \label}
\newcommand{\ee}{\end{equation}}
\newcommand{\bea}{\begin{eqnarray}\label}
\newcommand{\eea}{\end{eqnarray}}
\newcommand{\bas}{\begin{eqnarray*}}
\newcommand{\eas}{\end{eqnarray*}}
\newcommand{\bit}{\begin{itemize}}
\newcommand{\eit}{\end{itemize}}
\newcommand{\qed}{\hfill$\Box$ \vskip.2cm}
\newcommand{\nn}{\nonumber}
\newcommand{\R}{\mathbb{R}}
\newcommand{\N}{\mathbb{N}}
\newcommand{\pO}{\partial\Omega}

\newcommand{\eps}{\varepsilon}

\newcommand{\supp}{{\rm supp} \, }

\newcommand{\wto}{\rightharpoonup}
\newcommand{\wsto}{\stackrel{\star}{\rightharpoonup}}
\newcommand{\hra}{\hookrightarrow}
\newcommand{\io}{\int_\Omega}

\newcommand{\del}{\delta}
\newcommand{\al}{\alpha}

\newcommand{\sig}{\sigma}
\newcommand{\pa}{\partial}
\newcommand{\bom}{\overline{\Omega}}
\newcommand{\Om}{\Omega}

\newcommand{\ov}{\overline}

\newcommand{\hs}{\hspace*}
\newcommand{\sm}{\setminus}

\newcommand{\vp}{\varphi}
\newcommand{\lbal}{\left\{ \begin{array}{l}}
\newcommand{\lball}{\left\{ \begin{array}{ll}}
\newcommand{\ear}{\end{array} \right.}

\newcommand{\abs}{\\[5pt]}
\newcommand{\Abs}{\\[5mm]}

\newcommand{\adb}{\allowdisplaybreaks}

\newcommand{\ts}{t_\star}
\newcommand{\Kf}{K_f}

\newcommand{\oz}{\ov{z}}
\newcommand{\ueps}{u_\eps}
\newcommand{\veps}{v_\eps}

\newcommand{\feps}{f_\eps}

\newcommand{\zeps}{z_\eps}
\newcommand{\Teps}{\Theta_\eps}
\newcommand{\elleps}{\ell_\eps}
\newcommand{\vepsx}{v_{\eps x}}
\newcommand{\vepsxx}{v_{\eps xx}}

\newcommand{\vepsxxxx}{v_{\eps xxxx}}
\newcommand{\vepst}{v_{\eps t}}
\newcommand{\uepsx}{u_{\eps x}}
\newcommand{\uepsxx}{u_{\eps xx}}
\newcommand{\uepsxxx}{u_{\eps xxx}}
\newcommand{\uepst}{u_{\eps t}}
\newcommand{\Tepsx}{\Theta_{\eps x}}
\newcommand{\Tepsxx}{\Theta_{\eps xx}}
\newcommand{\Tepst}{\Theta_{\eps t}}
\newcommand{\zepsx}{z_{\eps x}}
\newcommand{\zepsxx}{z_{\eps xx}}
\newcommand{\zepst}{z_{\eps t}}
\newcommand{\epss}{\eps_\star}

\newcommand{\zd}{\zeta_\delta}
\newcommand{\Tinf}{\Theta_\infty}
\newcommand{\uinf}{u_\infty}
\begin{document}
\adb
\title{
Large time stabilization of rough-data solutions in one-dimensional nonlinear thermoelasticity}
\author{
Michael Winkler\footnote{michael.winkler@math.uni-paderborn.de}\\
{\small Universit\"at Paderborn, Institut f\"ur Mathematik}\\
{\small 33098 Paderborn, Germany} }
\date{}
\maketitle
\begin{abstract}
\noindent 
In an open bounded real interval $\Om$, the model for one-dimensional thermoelasticity given by
\bas
	\lball
	u_{tt} = u_{xx} - \big(f(\Theta)\big)_x,
	\qquad & x\in\Om, \ t>0, \\[1mm]
	\Theta_t = \Theta_{xx} - f(\Theta) u_{xt},
	\qquad & x\in\Om, \ t>0, 
	\ear
\eas
is considered along with homogeneous boundary conditions of Dirichlet type for $u$ and of Neumann type for $\Theta$,
under the assumption that $f\in C^1([0,\infty))$ satisfies $f(0)=0$, $f'\in L^\infty((0,\infty))$ and $f'>0$ on $[0,\infty)$.
The focus is on initial data which are merely required to be consistent with the fundamental principles of energy conservation
and entropy nondecrease, by satisfying 
\bas
	u_0\in W_0^{1,2}(\Om),
	\qquad 
	u_{0t} \in L^2(\Om)
	\qquad \mbox{and} \qquad
	0 \le \Theta_0\in L^1(\Om), \ \Theta_0 \not\equiv 0.
\eas
Despite an apparent lack of favorable compactness properties that have underlain previous related studies on more regular settings, 
it is shown that corresponding weak solutions, the existence of which is known from the literature, stabilize in the sense that
\bas
	\lim_{t\to\infty} \|u(\cdot,t)\|_{L^\infty(\Om)}=0
	\qquad \mbox{and} \qquad
	{\rm{ess}} \lim_{\hs{-4mm} t\to\infty} \|\Theta(\cdot,t)-\Tinf\|_{L^\infty(\Om)}=0
\eas
with some $\Tinf>0$.\abs
\end{abstract}
%
%
%
%
%
%
%
\newpage
\section{Introduction}\label{intro}
This manuscript is concerned with the initial-boundary value problem
\be{0}
	\lball
	u_{tt} = u_{xx} - \big(f(\Theta)\big)_x,
	\qquad & x\in\Om, \ t>0, \\[1mm]
	\Theta_t = \Theta_{xx} - f(\Theta) u_{xt},
	\qquad & x\in\Om, \ t>0, \\[1mm]
	u=0, \quad \Theta_x=0,
	\qquad & x\in\pO, \ t>0, \\[1mm]
	u(x,0)=u_0(x), \quad u_t(x,0)=u_{0t}(x), \quad \Theta(x,0)=\Theta_0(x),
	\qquad & x\in\Om,
	\ear
\ee
which is used to describe the spatio-temporal evolution of the deformation variable $u$ and the temperature $\Theta$ 
during conversion of mechanical energy into heat in one-dimensional thermoelastic materials.
In fact, in the case when $f\equiv id$ this system arises as a minimal representative among 
a comprehensive class of more general models for nonlinear thermoelasticity, specifically postulating the simple law
${\mathcal{E}}=\frac{1}{2} u_x^2 - \Theta u_x - \Theta \ln \Theta + \Theta$ 
for the free energy in a homogeneous one-dimensional body (\cite{hrusa_tarabek});
for more general functions $f$, (\ref{0}) can be viewed as a hyperbolic-parabolic approximation of a 
corresponding model associated with the choice 
${\mathcal{E}}=\frac{1}{2} u_x^2 - f(\Theta) u_x - \Theta \ln \Theta + \Theta$ within certain small-strain regimes
(\cite{hrusa_tarabek}, \cite{win_existence}).\abs
The mathematical analysis of (\ref{0}) and related systems for thermoelastic interaction has flourished already in the 1980s,
where papers by Slemrod, Dafermos and Hsiao from the 1980s (\cite{slemrod}, \cite{dafermos}, \cite{dafermos_hsiao_smooth})
have formed starting points for a series of works that have addressed issues of solvability and solution behavior
in various particular scenarios.
Common to a considerable part of these developments concerned with one-dimensional problems 
seems a restriction either to local-in-time solutions (\cite{racke88}),
or to the dynamics near constant states (see \cite{slemrod}, \cite{hrusa_tarabek}, 
\cite{racke_shibata}, \cite{racke_shibata_zheng}, \cite{kim} and \cite{jiang1990}, for instance);
in higher-dimensional domains, 
an apparently quite substantial limitation of available information on regularity 
goes along with a predominant concentration of corresponding literature on aspects of solvability
in appropriately generalized frameworks 
(\cite{roubicek}, \cite{mielke_roubicek}, \cite{owczarek_wielgos}, \cite{cieslak_muha_trifunovic},
\cite{roubicek_SIMA10}, \cite{blanchard_guibe}, \cite{gawinecki_zajaczkowski_CPAA}).
Even in the presence of additional dissipative mechanisms due to viscoelastic effects, 
a comprehensive understanding of qualitative behavior in situations far off equilibrium
seems to have been achieved only in one-dimensional settings (\cite{racke_zheng}, \cite{hsiao_luo}).\abs
{\bf Solutions emanating from large regular data.} \quad
For the purely thermoelastic problem (\ref{0}) with $f\equiv id$, 
only recently some substantial progress with regard to large-data solutions
could be achieved: As seen in \cite{cieslak_SIMA} and \cite{cieslak_MAAN}, namely, for initial data satisfying
\be{ireg}
	u_0\in W^{2,2}(\Om) \cap W_0^{1,2}(\Om),
	\qquad 
	u_{0t} \in W_0^{1,2}(\Om)
	\qquad \mbox{and} \qquad
	\Theta_0\in W^{1,2}(\Om)
	\mbox{ with $\Theta_0>0$ in } \bom,
\ee
a theory not only of global solvability but also of large time asymptotics can be built on a quasi-Lyapunov
inequality resulting from the identity
\be{en2}
	\frac{d}{dt} \bigg\{ \io \frac{\Theta_x^2}{\Theta} + \io u_{xt}^2 + \io u_{xx}^2 \bigg\}
	+ 2 \io \Theta (\ln\Theta)_{xx}^2 
	= \io \frac{\Theta_x^2}{\Theta} u_{xt}.
\ee
In particular, accordingly implied information on regularity of solutions, and especially on favorable
compactness properties of trajectories, has led the authors in \cite{cieslak_SIMA} and \cite{cieslak_MAAN} to the conclusion that
if the initial data comply with the hypotheses on regularity and positivity made in (\ref{ireg}), then
within some suitable framework of solvability a unique global solution indeed exists and satisfies
\be{conv}
	u(\cdot,t) \to 0
	\quad \mbox{in } W^{1,2}(\Om),
	\qquad
	u_t(\cdot,t) \to 0
	\quad \mbox{in } L^2(\Om),
	\qquad \mbox{and} \qquad
	\Theta(\cdot,t) \to \Theta_\infty
	\quad \mbox{in } L^2(\Om)
\ee
as $t\to\infty$, where 
$\Theta_\infty=\frac{1}{|\Om|} \cdot \big\{ \frac{1}{2} \io u_{0t}^2 + \frac{1}{2} \io u_{0x}^2 + \io \Theta_0 \big\}$.
These results extend classical knowledge on similar relaxation properties of weak solutions 
to linearized relatives, available actually also in higher-dimensional domains (\cite{koch}),
and on asymptotic stability of semitrivial states, even in stronger topological frameworks and for classes
of one-dimensional problems 
considerably more general than (\ref{0}) (\cite{racke_shibata}, \cite{racke_shibata_zheng}, \cite{hrusa_tarabek}).
By covering initial data of arbitrary size within the spaces of functions fulfilling (\ref{ireg}), 
this furthermore indicates a substantial trend toward equilibrium in (\ref{0}), also of trajectories starting from far apart,
and thereby draws a parallel to the well-understood situation in variants of (\ref{0}) that additionally account for
visoelastic damping
(\cite{racke_zheng}).\abs
{\bf Main results: Relaxation within energy- and entropy-maximal ranges of initial data.} \quad
From a physical perspective, two principles apparently more fundamental than that underlying (\ref{en2})
seem encrypted in the identities
\be{energy}
	\frac{d}{dt} \bigg\{ \frac{1}{2} \io u_t^2 + \frac{1}{2} \io u_x^2 + \io \Theta \bigg\}
	= 0
\ee
and
\be{entropy}
	\frac{d}{dt} \io \ell(\Theta) = - \io \frac{f'(\Theta)}{f^2(\Theta)} \Theta_x^2,
	\qquad \ell(\xi):=- \int_1^\xi \frac{d\sig}{f(\sig)}, \ \xi>0,
\ee
as formally associated with (\ref{0}) for arbitrary increasing positive functions $f$, 
These relations quantify the conservation of total energy and the nondecrease of entropy, respectively, and accordingly
it seems natural to consider (\ref{0}) within ranges of initial data which are maximal in this regard,
that is, to require that, only, 
\be{i99}
	u_0\in W_0^{1,2}(\Om),
	\qquad 
	u_{0t} \in L^2(\Om)
	\qquad \mbox{and} \qquad
	\Theta_0\in L^1(\Om)
	\mbox{ is such that $\Theta_0\ge 0$ a.e.~in } \Om.
\ee
It is, inter alia, a lack of smoothing in the hyperbolic subsystem of (\ref{0}) which firstly
gives rise to the evident methodological observation that a corresponding analysis 
can apparently no longer be based on (\ref{en2}) in such extended scenarios.
Secondly, aiming at a more application-oriented point of view
this moreover raises the question whether
discontinuities or further types of irregularities in the initial distributions may affect the asymptotic prevalence of
spatial homogeneity, despite the circumstance that
the deformation variable $u$ cannot be expected to undergo smoothing within finite time intervals;
classical findings on lack of regularity and even on blow-up phenomena in thermoelastic problems 
(\cite{dafermos_hsiao_BU}, \cite{racke_bu}, \cite{hrusa_messaoudi}) may serve as a caveat in this regard.\abs
The intention of the present manuscript consists in developing an approach capable of ruling out any such effect,
under the mere assumptions in (\ref{i99}) asserting large time decay of $u$ and stabilization of $\Theta$ toward 
a homogeneous limit state.
As a prerequisite, our considerations will rest on the following statement on global solvability in (\ref{0}),
as recently obtained in \cite{win_existence}.\Abs
{\bf Theorem A} \quad
{\it
  Let $\Om\subset\R$ be an open bounded interval, suppose that
  \be{f}
	f\in C^1([0,\infty))
	\mbox{ is such that $f(0)=0$ and $0 < f'(\xi) \le \Kf$ for all } \xi\ge 0
  \ee
  with some $\Kf>0$, and assume that
  \be{init}
	u_0\in W_0^{1,2}(\Om),
	\quad 
	u_{0t} \in L^2(\Om)
	\quad \mbox{and} \quad
	\Theta_0\in L^1(\Om)
	\mbox{ is nonnegative with $\Theta_0\not\equiv 0$.}
  \ee
  Then there exists at least one global weak solution $(u,\Theta)$ of (\ref{0}), in the sense of Definition \ref{dw} below, 
  which is such that actually
  \be{17.1}
	\lbal
	u \in C^0(\bom\times [0,\infty)) \cap L^\infty((0,\infty);W_0^{1,2}(\Om))
	\quad \mbox{with} \quad
	u_t \in L^\infty((0,\infty);L^2(\Om))
	\qquad \mbox{as well as} \\[1mm]
	\Theta \in C^0_w([0,\infty);L^1(\Om)) \cap \bigcap_{q\in [1,2)} L^q_{loc}([0,\infty);L^\infty(\Om)) \\
	\hs{20mm}
	\cap \bigcap_{q\in [1,\frac{4}{3})} L^q_{loc}([0,\infty);W^{1,2}(\Om)) \cap L^2_{loc}((0,\infty);W^{1,2}(\Om)),
	\ear
  \ee
  and that for
  all $\tau>0$ there exists $C(\tau)>0$ fulfilling
}
  \be{17.6}
	\frac{1}{C(\tau)} \le \Theta(x,t) \le C(\tau)
	\qquad \mbox{for a.e.~} (x,t)\in \Om\times (\tau,\infty).
  \ee
%
%
%
%
%
In addressing the question of how far the large time behavior in (\ref{0}) can be described in such settings,
we note that
the above reduction of requirements appears to bring about a considerable loss of compactness properties along trajectories:
The regularity information in (\ref{17.1}) and (\ref{17.6}) falls quite far short of temporally uniform
bounds for $(u,u_t,\Theta)$ in $W^{2,2}(\Om)\cap W^{1,2}(\Om)\times W^{1,2}(\Om)$ which result from (\ref{en2}) when $f\equiv id$
and (\ref{ireg}) is assumed (\cite{cieslak_MAAN}),
and due to the hyperbolic character of the mechanical part in (\ref{0})
it is to be expected that this limitation of knowledge, especially with regard to the solution component $u$,
is not only of technical nature but may rather reflect core properties of the problem itself.\abs
Despite these premises, our main results now will reveal that also under the minimal assumptions in (\ref{init})
a solution can always be found which indeed stabilizes toward homogeneity in the announced style,
thus particularly confirming that also solutions which remain rather irregular during evolution
may become smooth at least in the large time limit;
we emphasize that the following statement in this regard refers to $L^\infty$ topologies in both solution components
and hence particularly asserts spatially uniform convergence in the temperature variable, although Theorem A
does not even predict continuity of $\Theta$.
\begin{theo}\label{theo18}
  Suppose that $\Om\subset\R$ is an open bounded interval, that (\ref{f}) is satisfied with some $\Kf>0$, and that (\ref{init})
  holds. 
  Then there exists a global weak solution of (\ref{0}) which has the additional properties that (\ref{17.1}) holds,
  and that with some $\Tinf>0$ and some null set $N\subset (0,\infty)$,
  \be{18.1}
	u(\cdot,t) \to 0
	\quad \mbox{in } L^\infty(\Om)
	\qquad \mbox{as } t\to\infty,
  \ee
  and that
  \be{18.2}
	\Theta(\cdot,t) \to \Tinf
	\quad \mbox{in } L^\infty(\Om)
	\qquad \mbox{as } (0,\infty)\sm N \ni t\to\infty.
  \ee
\end{theo}
{\bf Challenges and ideas.} \quad
In line with the circumstance that the driving inhomogeneity $-\big(f(\Theta)\big)_x$ in the first equation in (\ref{0}) 
can in general not be expected to be favorably signed, a key toward an expedient analysis of large time relaxation 
seems to consist in an appropriate understanding of the dissipative action of diffusion in the temperature part.
Indeed, a quantifiable effect of this mechanism at a spatially global level is already expressed in (\ref{entropy}),
and a first step in our analysis will be devoted to turning the information on monotonicity of 
\be{01}
	0<t\mapsto \io z(\cdot,t),
	\qquad 
	z:=\ell(\Theta),
\ee
and on averaged decay of $z_x$ in the flavor of the inequality
\be{02}
	\int_1^\infty \io z_x^2 < \infty,
\ee
as thereby implied, into a statement on genuine stabilization of $\Theta$ according to (\ref{18.2}).
The preliminary Lemma \ref{lem67} in this regard will combine (\ref{02}) and a Poincar\'e-Sobolev inequality 
associated with the continuous embedding $W^{1,2}(\Om)\hra L^\infty(\Om)$ in a straightforward manner
to see that with $\Tinf>0$ uniquely determined by the relation $|\Om| \cdot \ell(\Tinf)=\lim_{t\to\infty} \io z(\cdot,t)$,
\be{03}
	\int_{t-1}^t \|\Theta(\cdot,s)-\Tinf\|_{L^\infty(\Om)} ds \to 0
	\qquad \mbox{as } t\to\infty.
\ee
In the rough-data setting concentrated on here, however, 
knowledge on regularity of $\Theta$ seems essentially limited to the $L^\infty$ bounds recorded in Theorem A,
and to thus remain insufficient for developing (\ref{03}) into the uniform convergence property in (\ref{18.2}) 
standard compactness arguments based on, e.g., the Arzel\`a-Ascoli theorem;
accordingly, the main step toward our derivation of (\ref{18.2}) will operate at regularity levels below continuity
and use a Moser-type recursion to supplement (\ref{03}) by an inequality of the form
\be{04}
	\|\Theta(\cdot,t)-\Tinf\|_{L^\infty(\Om)}
	\le C \|\Theta(\cdot,\ts)-\Tinf\|_{L^\infty(\Om)}
	+ C \cdot (t-\ts)^b,
	\qquad \mbox{for all $t\in (\ts,\ts+\frac{1}{2})$,}
\ee
which at the level of solutions to certain regularized variants of (\ref{0}) (see (\ref{0eps})) holds with some $b>0$ and $C>0$
whenever $\ts>1$ is such that $\|\Theta(\cdot,\ts)-\Tinf\|_{L^\infty(\Om)} \le 2^{-b}$ (Lemma \ref{lem61}).\abs
The second part of our analysis will thereafter be guided by the motivation to use the result on uniform stabilization of $\Theta$
thereby achieved to extract from a suitable approximate counterpart of the identity
\be{05}
	\pa_t (z-u_x) = z_{xx} - f'(\Theta) z_x^2
\ee
some information on large time decay of $u_x$. In fact, relying on the decay property of $z_x$ expressed in (\ref{02}), 
as well as on the fact that (\ref{18.2}) warrants uniform stabilization also of $z$, Lemma \ref{lem43}
will at its core argue toward transforming (\ref{05}) into a decay property of the form
\bas
	\int_{t_k}^{t_k+1} \|u_x(\cdot,t)-u_x(\cdot,t_k)\|_{(W^{1,2}(\Om)^\star}^2 dt \to 0
	\qquad \mbox{as } k\to\infty,
\eas
hence implying that
\bas
	\int_{t_k}^{t_k+1} \|u(\cdot,t)-u(\cdot,t_k)\|_{L^2(\Om)}^2 dt \to 0
	\qquad \mbox{as } k\to\infty,
\eas
whenever $(t_k)_{k\in\N}\subset (0,\infty)$ satisfies $t_k\to\infty$ as $k\to\infty$.
Utilizing this in the course of a suitable testing procedure applied to the first equation in (\ref{0}) will finally
identify all conceivable $\omega$-limits of $u$ as necessarily trivial.
\mysection{Preliminary results known from the literature}
In this section we briefly recall from \cite{win_existence} the solution concept referred to in Theorem A,
as well as some basic elements of the construction underlying the existence statement therein.\abs
Let us begin by importing from \cite{win_existence} the 
following basic notion of weak solvability which appears quite natural in the context of (\ref{0}).
\begin{defi}\label{dw}
  Let $\Om\subset\R$ be a bounded open interval, let $f\in C^0([0,\infty))$, and let
  $u_0\in L^1(\Om)$, $u_{0t}\in L^1(\Om)$ and $\Theta_0\in L^1(\Om)$.
  Then a pair $(u,\Theta)$ of functions
  \be{w1}
	\lbal
	u\in C^0([0,\infty);L^1(\Om)) \cap L^1_{loc}([0,\infty);W_0^{1,1}(\Om))
	\qquad \mbox{and} \\[1mm]
	\Theta\in L^1_{loc}([0,\infty);W^{1,1}(\Om))
	\ear
  \ee
  will be called a {\em global weak solution} of (\ref{0}) 
  if $\Theta\ge 0$ a.e.~in $\Om\times (0,\infty)$, 
  if $u_t, f'(\Theta) \Theta_x, f'(\Theta) \Theta_x u_t$ and $f(\Theta) u_t$ all belong to $L^1_{loc}(\bom\times [0,\infty))$,
  and if
  \bea{wu}
	\int_0^\infty \io u \vp_{tt}
	- \io u_{0t} \vp(\cdot,0)
	+ \io u_0 \vp_t(\cdot,0)
	= -  \int_0^\infty \io u_x \vp_x
	- \int_0^\infty \io f'(\Theta) \Theta_x \vp
  \eea
  for all $\vp\in C_0^\infty(\Om\times [0,\infty))$, as well as
  \bea{wt}
	- \int_0^\infty \io \Theta \vp_t - \io \Theta_0 \vp(\cdot,0)
	= - \int_0^\infty \io \Theta_x \vp_x
	+ \int_0^\infty \io f'(\Theta) \Theta_x u_t \vp
	+ \int_0^\infty \io f(\Theta) u_t \vp_x
  \eea
  for each $\vp\in C_0^\infty(\bom\times [0,\infty))$.
\end{defi}
As seen in \cite{win_existence}, a weak solution of (\ref{0}) can be obtained 
by means of a limit procedure involving the approximate problems
\be{0eps}
	\lball
	\vepst = - \eps \vepsxxxx + \uepsxx - \big( \feps(\Teps)\big)_x,
	\qquad & x\in\Om, \ t>0, \\[1mm]
	\uepst = \eps \uepsxx + \veps,
	\qquad & x\in\Om, \ t>0, \\[1mm]
	\Tepst = \Tepsxx - \feps(\Teps) \vepsx,
	\qquad & x\in\Om, \ t>0, \\[1mm]
	\veps=\vepsxx=0, \quad \ueps=0, \quad \Tepsx=0,
	\qquad & x\in\pO, \ t>0, \\[1mm]
	\veps(x,0)=v_{0\eps}(x), \quad \ueps(x,0)=u_{0\eps}(x), \quad \Teps(x,0)=\Theta_{0\eps}(x),
	\qquad & x\in\Om,
	\ear
\ee
where $(\feps)_{\eps\in (0,1)} \subset C^2([0,\infty))$ is such that
\be{fe1}
	0 \le \feps(\xi) \le 2\Kf \xi
	\quad \mbox{and} \quad
	0<\feps'(\xi) \le 2\Kf 
	\qquad \mbox{for all $\xi\ge 0$ and } \eps\in (0,1),
\ee
and that
\be{fe2}
	\feps \to f 
	\quad \mbox{in } C^1_{loc}([0,\infty))
	\qquad \mbox{as } \eps\searrow 0,
\ee
and where
$(v_{0\eps})_{\eps\in (0,1)} \subset C_0^\infty(\Om)$,
$(u_{0\eps})_{\eps\in (0,1)} \subset C_0^\infty(\Om)$
and $(\Theta_{0\eps})_{\eps\in (0,1)} \subset C^\infty(\bom)$
are such that for each $\eps\in (0,1)$ we have $\Theta_{0\eps}> 0$ in $\bom$ and
\be{ie1}
	\io \Theta_{0\eps} \ge \frac{1}{2} \io \Theta_0,
\ee
and that 
\be{ie}
	v_{0\eps} \to u_{0t}
	\quad \mbox{in } L^2(\Om),
	\qquad 
	u_{0\eps} \to u_0
	\quad \mbox{in } W^{1,2}(\Om)
	\qquad \mbox{and} \qquad
	\Theta_{0\eps} \to \Theta_0
	\quad \mbox{in } L^1(\Om)
\ee
as $\eps\searrow 0$. 
As has been seen in \cite[Lemmata 2.1 and 2.4]{win_existence}, all these problems indeed are globally solvable in the classical sense, 
and their solutions do not only satisfy the basic estimates (\ref{2.1})-(\ref{2.5}) that result from a corresponding
approximate relative of the energy identity (\ref{energy}), but moreover enjoy some further regularity features, 
among which we recall here the pointwise upper and lower bounds in (\ref{36.1}) and (\ref{36.2}) for explicit later reference.
\begin{lem}\label{lem0}
  For each $\eps\in (0,1)$, the problem (\ref{0eps}) admits a global classical solution $(\veps,\ueps,\Teps)$ with
  \bas
	\lbal
	\veps\in C^{4,1}([\bom\times [0,\infty)), \\
	\ueps\in C^{2,1}(\bom\times [0,\infty))
	\qquad \mbox{and} \\
	\Teps \in C^{2,1}(\bom\times [0,\infty)),
	\ear
  \eas
  such that there exists $C>0$ such that 
  \be{2.1}
	\io \veps^2(\cdot,t) \le C
	\qquad \mbox{for all $t>0$ and } \eps\in (0,1)
  \ee
  and
  \be{2.2}
	\io \uepsx^2(\cdot,t) \le C
	\qquad \mbox{for all $t>0$ and } \eps\in (0,1)
  \ee
  as well as
  \be{2.3}
	\io \Teps(\cdot,t) \le C
	\qquad \mbox{for all $t>0$ and } \eps\in (0,1),
  \ee
  and that
  \be{2.4}
	\eps \int_0^\infty \io \vepsxx^2 \le C
	\qquad \mbox{for all } \eps\in (0,1)
  \ee
  and
  \be{2.5}
	\eps \int_0^\infty \io \uepsxx^2 \le C
	\qquad \mbox{for all } \eps\in (0,1).
  \ee
  Given any $\tau>0$, we moreover have
  \be{36.1}
	\sup_{\eps\in (0,1)} \sup_{(x,t)\in \Om\times (\tau,\infty)} \Teps(x,t) < \infty,
  \ee
  and we can find $\epss=\epss(\tau)\in (0,1)$ such that
  \be{36.2}
	\inf_{\eps\in (0,\epss)} \inf_{(x,t)\in \Om\times (\tau,\infty)} \Teps(x,t) >0.
  \ee
\end{lem}
In fact, in \cite[Lemmata 4.1, 7.3 and 8.1]{win_existence} these estimates have been found to imply the following
refined version of the statement in Theorem A. 
Here we particularly highlight the strong convergence feature in (\ref{51.7}) obtained in
\cite[Lemma 7.3]{win_existence}: While apparently going beyond what might naturally be expected to result from 
(\ref{2.1}), this property will be of crucial importance in our verification of the fact that
also the weak approximation feature in (\ref{51.9}) can be converted into a corresponding statement 
on strong convergence, to be accomplished in Lemma \ref{lem65}.
\begin{lem}\label{lem51}
  There exist $(\eps_j)_{j\in\N} \subset (0,1)$, a null set $N\subset (0,\infty)$ 
  and a global weak solution $(u,\Theta)$ of (\ref{0}) such that (\ref{17.1}) and (\ref{17.6}) hold,
  that $\eps_j\searrow 0$ as $j\to\infty$, and that the corresponding solutions of (\ref{0eps}) satisfy
  \begin{eqnarray}
	& & \ueps\to u
	\qquad \mbox{in } C^0_{loc}(\bom\times [0,\infty)),
	\label{51.5} \\
	& & \veps \to u_t
	\qquad \mbox{in } L^2_{loc}(\bom\times [0,\infty)),
	\label{51.7}, \\
	& & \Teps \to \Theta
	\qquad \mbox{a.e.~in $\Om\times (0,\infty)$},
	\label{51.8} \\
	& & \Teps(\cdot,t) \to \Theta(\cdot,t)
	\qquad \mbox{in $C^0(\bom)$ \quad for all $t\in (0,\infty)\sm N$}
	\qquad \qquad \mbox{and}
	\label{51.87} \\
	& & \Tepsx \wto \Theta_x
	\qquad \mbox{in } L^2_{loc}(\bom\times (0,\infty))
	\label{51.9} 
  \end{eqnarray}
  as $\eps=\eps_j\searrow 0$.
\end{lem}
As a final preliminary recorded in \cite[Lemmata 6.1 and 6.2]{win_existence}, 
we include a basic observation concerned with the entropy evoution property from (\ref{entropy}) 
when translated to the setting of (\ref{0eps}).
\begin{lem}\label{lem4}
  For $\eps\in (0,1)$, let $\veps$ and $\Teps$ be as in Lemma \ref{lem0}. Then the function defined by
  \be{zeps}
	\zeps(x,t):=\elleps(\Teps(x,t)),
	\quad x\in\bom, \ t\ge 0,
	\qquad \mbox{with} \qquad
	\elleps(\xi):= - \int_1^\xi \frac{d\sig}{\feps(\sig)},
	\quad \xi>0,
  \ee
  is an element of $C^{2,1}(\bom\times [0,\infty))$ and solves
  \be{0zeps}
	\lball
	\zepst = \zepsxx - \feps'(\Teps) \zepsx^2 + \vepsx,
	\qquad & x\in\Om, \ t>0, \\[1mm]
	\zeps=0,
	\qquad & x\in\pO, \ t>0, \\[1mm]
	\zeps(x,0)=\elleps(\Theta_{0\eps}),
	\qquad & x\in\Om,
	\ear
  \ee
  in the classical sense. Moreover,
  \be{5.1}
	\frac{d}{dt} \io \zeps = - \io \feps'(\Teps) \zepsx^2
	\qquad \mbox{for all } t>0.
  \ee
\end{lem}
\mysection{Stabilization of $\Theta$}
Now the first substantial step toward our large time analysis uses the approximate entropy identity in (\ref{5.1})
and exploits both the temporal monotonicity property and the dissipation process expressed therein.
In light of our knowledge on two-sided bounds for $\Teps$ documented in (\ref{36.1}) and (\ref{36.2}), 
namely, this yields a unique constant function
$\Tinf$ that appears as the only candidate for $\omega$-limits in the following, here yet temporally averaged, sense.
%
\begin{lem}\label{lem67}
  Let $\Theta$ be as obtained in Lemma \ref{lem51}. 
  Then there exists $\Theta_\infty>0$ such that
  \be{67.1}
	\int_{t-1}^t \|\Theta(\cdot,s)-\Tinf\|_{L^\infty(\Om)} ds \to 0
	\qquad \mbox{as } t\to\infty.
  \ee
  Moreover, the function defined by letting
  \be{z}
	z(x,t):=\ell(\Theta(x,t)),
	\quad (x,t)\in\Om\times (1,\infty),
	\qquad \mbox{with} \qquad
	\ell(\xi):=-\int_1^\xi \frac{d\sig}{f(\sig)},
	\quad \xi>0,
  \ee
  satisfies
  \be{67.01}
	z \in L^\infty(\Om\times (1,\infty)) \cap L^2_{loc}([1,\infty);W^{1,2}(\Om))
  \ee
  with
  \be{67.111}
	\int_1^\infty \io z_x^2 < \infty.
  \ee
\end{lem}
\proof
  Using (\ref{36.2}) and (\ref{36.1}), we can find $\epss\in (0,1)$, $c_1>0$ and $c_2>0$ such that
  \be{67.2}	
	c_1 \le \Teps(x,t) \le c_2
	\qquad \mbox{for all $x\in\Om, t>1$ and } \eps\in (0,\epss),
  \ee
  which by (\ref{5.1}), in view of (\ref{fe2}) and the strict positivity of $f'$ on $(0,\infty)$, implies that with some $c_3>0$
  we have $\frac{d}{dt} \io \zeps \le - c_3\io \zepsx^2$ for all $t>1$ and $\eps\in (0,\epss)$, and hence
  \bas
	\io \zeps(\cdot,t) + c_3 \int_{\ts}^t \io \zepsx^2
	\le \io \zeps(\cdot,\ts)
	\qquad \mbox{for all $\ts>1, t>\ts$ and } \eps\in (0,\epss).
  \eas
  Thus, if we let the null set $N\subset (0,\infty)$ be as provided by Lemma \ref{lem51}, then since 
  $\zepsx \wto z_x$ in $L^2_{loc}(\bom\times [1,\infty))$
  as $\eps=\eps_j\searrow 0$ by (\ref{67.2}) with (\ref{51.9}), 
  relying on lower semicontinuity of norms in $L^2$ spaces with respect
  to weak convergence we see that (\ref{67.01}) holds, and that
  \be{67.3}
	\io z(\cdot,t) + c_3 \int_{\ts}^t \io z_x^2 \le \io z(\cdot,\ts)
	\qquad \mbox{for all $\ts\in (1,\infty)\sm N$ and } t\in (\ts,\infty)\sm N.
  \ee
  In particular, defining $\oz(t):=\frac{1}{|\Om|} \io z(\cdot,t)$ for $t\in (1,\infty)\sm N$ we obtain a nonincreasing function
  $\oz$ which thanks to (\ref{67.2}) maps $(1,\infty)\sm N$ into $[\ell(c_2),\ell(c_1)]$, 
  and which thanks to a Poincar\'e inequality has the property that with some $c_4>0$,
  \bas
	\|z(\cdot,t)-\oz(t)\|_{L^\infty(\Om)} \le c_4 \|z_x(\cdot,t)\|_{L^2(\Om)}
	\qquad \mbox{for all } t\in (1,\infty)\sm N.
  \eas
  Therefore, (\ref{67.3}) implies that if we let $z_\infty\in\R$ be such that
  \be{67.4}
	\oz(t)\searrow z_\infty
	\qquad \mbox{as } (1,\infty)\sm N \ni t \to \infty,
  \ee
  then by the Cauchy-Schwarz inequality we find that
  for all $\ts\in (1,\infty)\sm N$ and $t_0\in (\ts,\infty)\sm N$,
  \bas
	\int_{\ts}^{t_0} \|z(\cdot,s)-z_\infty\|_{L^\infty(\Om)} ds
	&\le& \int_{\ts}^{t_0} \|z(\cdot,s)-\oz(s)\|_{L^\infty(\Om)} ds
	+ \int_{\ts}^{t_0} (\oz(s)- z_\infty) ds \\
	&\le& c_4 \int_{\ts}^{t_0} \|z_x(\cdot,s)\|_{L^2(\Om)} ds
	+ (\oz(\ts)-z_\infty) \cdot (t_0-\ts) \\
	&\le& c_4 \cdot \bigg\{ \int_{\ts}^{t_0} \io z_x^2 \bigg\}^\frac{1}{2} \cdot (t_0-\ts)^\frac{1}{2}
	+ (\oz(\ts)-z_\infty) \cdot (t_0-\ts) \\
	&\le& c_3^{-\frac{1}{2}} c_4 |\Om|^{-\frac{1}{2}} \cdot (\oz(\ts)-\oz(t_0))^\frac{1}{2} (t_0-\ts)^\frac{1}{2}
	+ (\oz(\ts)-z_\infty) \cdot (t_0-\ts) \\
	&\le& c_3^{-\frac{1}{2}} c_4 |\Om|^{-\frac{1}{2}} \cdot (\oz(\ts)-z_\infty)^\frac{1}{2} (t_0-\ts)^\frac{1}{2}
	+ (\oz(\ts)-z_\infty) \cdot (t_0-\ts).
  \eas   
  According to the fact that $(t-2,t-1)\sm N\ne\emptyset$ and $(t,t+1)\sm N\ne\emptyset$ for all $t>3$, due to (\ref{67.4}) this
  implies that
  \bas
	\int_{t-1}^t \|z(\cdot,s)-z_\infty\|_{L^\infty(\Om)} ds \to 0
	\qquad \mbox{as } (3,\infty)\ni t \to\infty,
  \eas
  so that since the strict monotonicity of $\ell$ ensures the existence of $\Tinf\in [c_1,c_2]$ such that $\ell(z_\infty)=\Tinf$,
  (\ref{67.1}) follows upon observing that for a.e.~$(x,t)\in \Om\times (1,\infty)$ we can find $\xi(x,t)\in [c_1,c_2]$ such that
  \bas
	|z(x,t)-z_\infty|
	= \big| \ell(\Theta(x,t))-\ell(\Tinf) \big|
	= \Big| \frac{1}{f(\xi(x,t))} \Big| \cdot |\Theta(x,t)-\Tinf|,
  \eas
  and that here $\big|\frac{1}{f(\xi(x,t))}| \ge \frac{1}{f(c_2)}$ by monotonicity of $f$.
  The integrability property in (\ref{67.111}), finally, is a direct consequence of (\ref{67.3}) and the boundedness of $\oz$.
\qed
Transforming the above into genuine convergence with respect to the norm in $L^\infty(\Om)$ will rely on the following
outcome of a Moser-type recursive argument, revealing that presupposed smallness of $\Teps-a$ at some time and for some $a\ge 0$
is essentially conserved at least throughout an adjacent time interval of favorably controllable length. 
\begin{lem}\label{lem61}
  There exists $b>0$ such that for each $a\ge 0$ one can find $C(a)>0$ with the property that whenever $\eps\in (0,1)$ and
  $\ts>1$ are such that for the solution of (\ref{0eps}) we have
  \be{61.1}
	\|\Teps(\cdot,\ts)-a\|_{L^\infty(\Om)}^\frac{1}{b} \le \frac{1}{2},
  \ee
  it follows that
  \be{61.2}
	\|\Teps(\cdot,t)-a\|_{L^\infty(\Om)}
	\le C(a) \|\Teps(\cdot,\ts)-a\|_{L^\infty(\Om)}
	+ C(a) \cdot (t-\ts)^b
	\qquad \mbox{for all $t\in (\ts,\ts+\frac{1}{2})$.}
  \ee
\end{lem}
\proof
  We fix any $b\in (0,\frac{1}{7})$ and then obtain that
  \be{61.3}
	p_k:=2^k
	\quad \mbox{and} \quad
	\al_k:=(2^k-1)b,
	\qquad k\ge 0,
  \ee
  satisfy 
  \be{61.4}
	\al_k=2\al_{k-1} + b 
	\qquad \mbox{for all } k\ge 1
  \ee
  and hence
  \bea{61.5}
	\frac{2p_k-4}{p_k+4} \cdot \al_{k-1} + 1 - \al_k
	&=& \Big(\frac{2p_k-4}{p_k+4} - 2 \Big)\cdot\al_{k-1} + 1 - b
	= \frac{-12}{2^k+4} \cdot (2^{k-1}-1)b + 1-b  \nn\\
	&=& \frac{-12 \cdot 2^{k-1} + 12 - 2^k - 4}{2^k+4} \cdot b + 1 
	= \frac{-7+8\cdot 2^{-k}}{1+4\cdot 2^{-k}} \cdot b + 1 \nn\\
	&\ge& \frac{-7}{1+4\cdot 2^{-k}} \cdot b + 1 
	\ge -7b+1 \ge 0
	\qquad \mbox{for all } k\ge 1.
  \eea
  We moreover employ Lemma \ref{lem0} to find $c_1>0$ and $c_2>0$ such that
  \be{61.67}
	\io \Teps \le c_1
	\quad \mbox{and} \quad
	\io \veps^2 \le c_2
	\qquad \mbox{for all $t>0$ and } \eps\in (0,1),
  \ee
  and henceforth assuming that $a\ge 0$, $\eps\in (0,1)$ and $\ts>1$ are such that (\ref{61.1}) holds, we let
  \bas
	d(x,t):=\Teps(x,t)-a,
	\quad \mbox{$(x,t)\in\bom\times [\ts,\infty)$}
	\qquad \mbox{and} \qquad
	\del:=\|\Teps(\cdot,\ts)-a\|_{L^\infty(\Om)}^\frac{1}{b},
  \eas
  observing that due to (\ref{61.3}) and (\ref{61.67}), for
  \be{61.8}
	M_k:=1+\sup_{t\in (\ts,\ts+\frac{1}{2})} \bigg\{ (t-\ts+\del)^{-\al_k} \cdot \io |d|^{p_k}(\cdot,t) \bigg\},
	\qquad k\ge 0,
  \ee
  we have
  \be{61.9}
	M_0
	= 1 + \sup_{t\in (\ts,\ts+\frac{1}{2})} \io |\Teps(\cdot,t)-a|
	\le 1 + (c_1+a)|\Om|.
  \ee
  To estimate $M_k$ for $k\ge 1$, noting that then $p_k$ is even we can use (\ref{0eps}) to see that thanks to Young's inequality,
  (\ref{61.67}) and (\ref{fe1}),
  \bea{61.10}
	\hs{-6mm}
	\frac{d}{dt} \io d^{p_k}	
	&=& - p_k(p_k-1) \io d^{p_k-2} d_x^2
	+ p_k \io d^{p_k-1} \feps'(\Teps) d_x \veps
	+ p_k(p_k-1) \io d^{p_k-2} \feps(\Teps) d_x \veps \nn\\
	&\le& - \frac{p_k(p_k-1)}{2} \io d^{p_k-2} d_x^2
	+ \frac{p_k}{p_k-1} \io d^{p_k} \feps'^2(\Teps) \veps^2
	+ p_k(p_k-1) \io d^{p_k-2} \feps^2(\Teps) \veps^2 \nn\\
	&\le& - \frac{p_k(p_k-1)}{2} \io d^{p_k-2} d_x^2
	+ 4 \Kf^2 p_k^2 \io d^{p_k} \veps^2
	+ 4 \Kf^2 p_k^2 \io d^{p_k-2} \Teps^2 \veps^2 \nn\\
	&\le& - \frac{p_k(p_k-1)}{2} \io d^{p_k-2} d_x^2
	+ 4 c_2 \Kf^2 p_k^2 \|d^{p_k}\|_{L^\infty(\Om)}
	+ 4 c_2 \Kf^2 p_k^2 \|d^{p_k-2} \Teps^2\|_{L^\infty(\Om)} \nn\\
	&\le& - \frac{p_k(p_k-1)}{2} \io d^{p_k-2} d_x^2
	+ 4 c_2 \Kf^2 p_k^2 \|d^{p_k}\|_{L^\infty(\Om)} \nn\\
	& & + 8 c_2 \Kf^2 p_k^2 \|d^{p_k-2} (\Teps-a)^2\|_{L^\infty(\Om)} 
	+ 8 c_2 a^2 \Kf^2 p_k^2 \|d^{p_k-2}\|_{L^\infty(\Om)} \nn\\
	&\le& - \frac{p_k(p_k-1)}{2} \io d^{p_k-2} d_x^2
	+ c_3 p_k^2 \|d^\frac{p_k}{2}\|_{L^\infty(\Om)}^2
	+ c_3 p_k^2 \|d^\frac{p_k}{2}\|_{L^\infty(\Om)}^\frac{2p_k-4}{p_k}
	\qquad \mbox{for all } t>\ts
  \eea
  with $c_3\equiv c_3(a):=12c_2 \Kf^2 + 8c_2 a^2 \Kf^2$, because $\frac{p_k}{p_k-1} \le p_k^2$ and $p_k(p_k-1)\le p_k^2$,
  and because $\Teps^2 \le 2 (\Teps-a)^2 + 2a^2$.
  We now employ a Gagliardo-Nirenberg inequality to find $c_4\ge 1$ such that
  \bas
	\|\psi\|_{L^\infty(\Om)}
	\le c_4 \|\psi_x\|_{L^2(\Om)}^\frac{2}{3} \|\psi\|_{L^1(\Om)}^\frac{1}{3}
	+ c_4 \|\psi\|_{L^1(\Om)}
	\qquad \mbox{for all } \psi\in C^1(\bom),
  \eas
  and use this twice along with Young's inequality in estimating
  \bea{61.11}
	c_3 p_k^2 \|d^\frac{p_k}{2}\|_{L^\infty(\Om)}^2
	&\le& 2c_3 c_4^2 p_k^2 \big\| (d^\frac{p_k}{2})_x\big\|_{L^2(\Om)}^\frac{4}{3} \|d^\frac{p_k}{2}\|_{L^1(\Om)}^\frac{2}{3}
	+ 2c_3 c_4^2 p_k^2 \|d^\frac{p_k}{2}\|_{L^1(\Om)}^2 \nn\\
	&=& \Big\{ \frac{1}{2} \big\| (d^\frac{p_k}{2})_x\big\|_{L^2(\Om)}^2 \Big\}^\frac{2}{3} \cdot
		2^\frac{5}{3} c_3 c_4^2 p_k^2 \|d^\frac{p_k}{2}\|_{L^1(\Om)}^\frac{2}{3}
	+ 2c_3 c_4^2 p_k^2 \|d^\frac{p_k}{2}\|_{L^1(\Om)}^2 \nn\\
	&\le& \frac{1}{2} \big\| (d^\frac{p_k}{2})_x\big\|_{L^2(\Om)}^2
	+ 32 c_3^3 c_4^6 p_k^6 \|d^\frac{p_k}{2}\|_{L^1(\Om)}^2
	+ 2c_3 c_4^2 p_k^2 \|d^\frac{p_k}{2}\|_{L^1(\Om)}^2 \nn\\
	&\le& \frac{p_k^2}{8} \io d^{p_k-2} d_x^2
	+ c_5 p_k^6 \|d^\frac{p_k}{2}\|_{L^1(\Om)}^2
	\qquad \mbox{for all } t>\ts
  \eea
  and
  \bea{61.12}
	c_3 p_k^2 \|d^\frac{p_k}{2}\|_{L^\infty(\Om)}^\frac{2p_k-4}{p_k}
	&\le& (2c_4)^\frac{2p_k-4}{p_k} c_3 p_k^2 \big\| (d^\frac{p_k}{2})_x\big\|_{L^2(\Om)}^\frac{4p_k-8}{3p_k} 
		\|d^\frac{p_k}{2}\|_{L^1(\Om)}^\frac{2p_k-4}{3p_k}
	+ (2c_4)^\frac{2p_k-4}{p_k} c_3 p_k^2 \|d^\frac{p_k}{2}\|_{L^1(\Om)}^\frac{2p_k-4}{p_k} \nn\\
	&=& \Big\{ \frac{1}{2} \big\| (d^\frac{p_k}{2})_x\big\|_{L^2(\Om)}^2 \Big\}^\frac{2p_k-4}{3p_k} \cdot
		2^\frac{2p_k-4}{3p_k} \cdot (2c_4)^\frac{2p_k-4}{p_k} c_3 p_k^2 \|d^\frac{p_k}{2}\|_{L^1(\Om)}^\frac{2p_k-4}{3p_k}
		\nn\\
	& & + (2c_4)^\frac{2p_k-4}{p_k} c_3 p_k^2 \|d^\frac{p_k}{2}\|_{L^1(\Om)}^\frac{2p_k-4}{p_k} \nn\\
	&\le& \frac{1}{2} \big\| (d^\frac{p_k}{2})_x\big\|_{L^2(\Om)}^2
	+ 2^\frac{2p_k-4}{p_k+4} \cdot (2c_4)^\frac{6p_k-12}{p_k+4} c_3^\frac{3p_k}{p_k+4} p_k^\frac{6p_k}{p_k+4}
		\|d^\frac{p_k}{2}\|_{L^1(\Om)}^\frac{2p_k-4}{p_k+4} \nn\\
	& & + (2c_4)^\frac{2p_k-4}{p_k} c_3 p_k^2 \|d^\frac{p_k}{2}\|_{L^1(\Om)}^\frac{2p_k-4}{p_k} \nn\\
	&\le& \frac{p_k^2}{8} \io d^{p_k-2} d_x^2
	+ c_6 p_k^6 \|d^\frac{p_k}{2}\|_{L^1(\Om)}^\frac{2p_k-4}{p_k+4} \nn\\
	& & + c_7 p_k^6 \|d^\frac{p_k}{2}\|_{L^1(\Om)}^\frac{2p_k-4}{p_k}
	\qquad \mbox{for all } t>\ts
  \eea
  with $c_5\equiv c_5(a):=32 c_3^3 c_4^6 + 2c_3 c_4^2$,
  $c_6\equiv c_6(a):=4\cdot (2c_4)^6 \cdot \max\{1,c_3^3\}$ and $c_7\equiv c_7(a):=(2c_4)^2 c_3$,
  because $\frac{2p_k-4}{p_k+4} \le 2, \frac{6p_k-12}{p_k+4} \le 6, \frac{3p_k}{p_k+4} \le 3$ and $\frac{2p_k-4}{p_k} \le 2$,
  and because $2c_4\ge 1$.
  As $\frac{2p_k-4}{p_k+4} \le \frac{2p_k-4}{p_k} \le 2$ and hence
  \bas
	 \|d^\frac{p_k}{2}\|_{L^1(\Om)}^\frac{2p_k-4}{p_k}
	\le  \|d^\frac{p_k}{2}\|_{L^1(\Om)}^2 +  \|d^\frac{p_k}{2}\|_{L^1(\Om)}^\frac{2p_k-4}{p_k+4}
	\qquad \mbox{for all } t>\ts
  \eas
  by Young's inequality, in view of the fact that $\frac{p_k(p_k-1)}{2} \ge \frac{p_k^2}{4}=2\cdot \frac{p_k^2}{8}$ we infer
  from (\ref{61.10}), (\ref{61.11}) and (\ref{61.12}) that with $c_8\equiv c_8(a):=c_5+c_7$ and $c_9\equiv c_9(a):=c_6+c_7$,
  \bas
	\frac{d}{dt} \io d^{p_k}
	\le c_8 p_k^6 \|d^\frac{p_k}{2}\|_{L^1(\Om)}^2
	+ c_9 p_k^6 \|d^\frac{p_k}{2}\|_{L^1(\Om)}^\frac{2p_k-4}{p_k+4}
	\qquad \mbox{for all } t>\ts,
  \eas
  where in view of (\ref{61.8}) and the identity $\frac{p_k}{2}=p_{k-1}$, for all $t\in (\ts,\ts+\frac{1}{2})$ we have
  \bas
	\|d^\frac{p_k}{2}\|_{L^1(\Om)}^2 \le (t-\ts+\del)^{2\al_{k-1}} M_{k-1}^2
  \eas
  and
  \bas
	\|d^\frac{p_k}{2}\|_{L^1(\Om)}^\frac{2p_k-4}{p_k+4}
	\le (t-\ts+\del)^{\frac{2p_k-4}{p_k+4}\cdot \al_{k-1}} M_{k-1}^\frac{2p_k-4}{p_k+4}
	\le (t-\ts+\del)^{\frac{2p_k-4}{p_k+4}\cdot \al_{k-1}} M_{k-1}^2,
  \eas
  as $M_{k-1} \ge 1$. Therefore,
  \bas
	\frac{d}{dt} \io d^{p_k} 
	\le c_8 p_k^6 (t-\ts+\del)^{2\al_{k-1}} M_{k-1}^2
	+ c_9 p_k^6 (t-\ts+\del)^{\frac{2p_k-4}{p_k+4}\cdot\al_{k-1}} M_{k-1}^2
	\qquad \mbox{for all $t\in (\ts,\ts+\frac{1}{2})$},
  \eas
  which upon an integration shows that for all $t>\ts$,
  \bea{61.111}
	\io d^{p_k}(\cdot,t)
	\le \io d^{p_k}(\cdot,\ts)
	+ c_8 p_k^6 (t-\ts+\del)^{2\al_{k-1}+1} M_{k-1}^2
	+ c_9 p_k^6 (t-\ts+\del)^{\frac{2p_k-4}{p_k+4}\cdot\al_{k-1}+1} M_{k-1}^2,
  \eea
  since 
  \bas
	\int_{\ts}^t (s-\ts+\del)^\iota
	\le (t-\ts+\del)^\iota \cdot (t-\ts) 
	\le (t-\ts+\del)^{\iota+1}
	\qquad \mbox{for all $t>\ts$ and } \iota>0.
  \eas
  We now rely on the smallness assumption in (\ref{61.1}), which namely means that $\del\le\frac{1}{2}$, and that for
  $t\in (\ts,\ts+\frac{1}{2})$ we may utilize (\ref{61.4}) and (\ref{61.5}) to see that
  \bas
	(t-\ts+\del)^{2\al_{k-1}+1} = (t-\ts+\del)^{\al_k} (t-\ts+\del)^{1-b} \le (t-\ts+\del)^{\al_k}
  \eas
  and
  \bas
	(t-\ts+\del)^{\frac{2p_k-4}{p_k+4}\cdot \al_{k-1}+1}
	= (t-\ts+\del)^{\al_k} (t-\ts+\del)^{\frac{2p_k-4}{p_k+4}\cdot \al_{k-1} +1-\al_k}
	\le (t-\ts+\del)^{\al_k}.
  \eas
  As (\ref{61.1}) together with (\ref{61.3}) and the inequalities $\del\le 1$, $M_{k-1}\ge 1$ and $p_k\ge 1$ furthermore ensures that
  \bas
	\io d^{p_k}(\cdot,\ts)
	&\le& |\Om| \cdot \|d(\cdot,\ts)\|_{L^\infty(\Om)}^{p_k}
	= |\Om| \del^{bp_k}
	= |\Om| (\del^{\al_k})^\frac{bp_k}{\al_k}
	= |\Om| (\del^{\al_k})^\frac{2^k}{2^k-1} \\
	&\le& |\Om| \del^{\al_k}
	\le |\Om| p_k^6 (t-\ts+\del)^{\al_k} M_{k-1}^2
	\qquad \mbox{for all } t>\ts,
  \eas
  from (\ref{61.111}) we infer that
  \bas
	\io d^{p_k}(\cdot,t) \le c_{10} p_k^6 (t-\ts+\del)^{\al_k} M_{k-1}^2
	\qquad \mbox{for all $t\in (\ts,\ts+\frac{1}{2})$}
  \eas
  with $c_{10}\equiv c_{10}(a):=|\Om|+ c_8 + c_9$.
  Hence, writing $c_{11}\equiv c_{11}(a):=64(1+c_{10})$ we find that
  \bas
	M_k \le 1 + c_{10} p_k^6 M_{k-1}^2
	\le (1+c_{10})^k \cdot 2^{6k} M_{k-1}^2
	= c_{11}^k M_{k-1}^2
	\qquad \mbox{for all } k\ge 1,
  \eas
  which by a straightforward induction implies that 
  \bas
	M_k \le c_{11}^{2^{k+1}-k-2} M_0^{2^k}
	\qquad \mbox{for all } k\ge 1.
  \eas
  Therefore, $M_k^{\frac{1}{2^k}} \le c_{11}^2 M_0$ for all $k\ge 1$, so that again in view of (\ref{61.8}),
  \bas
	\|d(\cdot,t)\|_{L^{p_k}(\Om)}
	\le c_{11}^2 M_0 \cdot (t-\ts+\del)^\frac{\al_k}{p_k}
	\qquad \mbox{for all $t\in (\ts,\ts+\frac{1}{2})$}.
  \eas
  Since $\frac{\al_k}{p_k} \to b$ as $k\to\infty$, this entails that
  \bas
	\|d(\cdot,t)\|_{L^\infty(\Om)}
	\le c_{11}^2 M_0 \cdot (t-\ts+\del)^b
	\le 2^b c_{11}^2 M_0 \cdot \big\{ (t-\ts)^b + \del^b\big\}
	\qquad \mbox{for all $t\in (\ts,\ts+\frac{1}{2})$},
  \eas
  which due to (\ref{61.9}) yields (\ref{61.2}) with $C(a):=2^b c_{11}^2 \cdot \big(1+(c_1+a)|\Om|\big)$.
\qed
In fact, combining the precedent two lemmata yields genuine stabilization in the temperature variable:
\begin{lem}\label{lem68}
  If the function $\Theta$ and the null set $N\subset (0,\infty)$ are as in Lemma \ref{lem51}, then the  
  number $\Tinf$ from Lemma \ref{lem67} has the property that
  \be{68.1}
	\Theta(\cdot,t) \to \Tinf
	\quad \mbox{in } L^\infty(\Om)
	\qquad \mbox{as } (0,\infty) \sm N \ni t\to\infty.
  \ee
\end{lem}
\proof
  We suppose that we could find $c_1>0$ and $(t_k)_{k\in\N}\subset (2,\infty)\sm N$ such that $t_k\to\infty$ as $k\to\infty$, but that
  \be{68.2}
	\|\Theta(\cdot,t_k)-\Tinf\|_{L^\infty(\Om)} \ge c_1
	\qquad \mbox{for all } k\in\N,
  \ee
  and to derive a contradiction from this, we employ Lemma \ref{lem61} to fix $b>0$ and $c_2>0$ with the property that if
  $\eps\in (0,1)$ and $\ts>1$ are such that $\|\Teps(\cdot,\ts)-\Tinf\|_{L^\infty(\Om)}^\frac{1}{b}\le \frac{1}{2}$, then
  \be{68.3}
	\|\Teps(\cdot,t)-\Tinf\|_{L^\infty(\Om)} \le c_2\|\Teps(\cdot,\ts)-\Tinf\|_{L^\infty(\Om)} + c_2 (t-\ts)^b
	\qquad \mbox{for all $t\in (\ts,\ts+\frac{1}{2})$.}
  \ee
  We thereupon let $\tau\in (0,\frac{1}{2})$ be such that
  \be{68.4}
	c_2 \tau^b \le \frac{c_1}{4},
  \ee
  take $\eta>0$ such that both
  \be{68.5}
	\Big(\frac{2\eta}{\tau}\Big)^\frac{1}{b} \le \frac{1}{2}
	\qquad \mbox{and} \qquad 
	c_2 \cdot \frac{2\eta}{\tau} \le \frac{c_1}{4},
  \ee
  and invoke Lemma \ref{lem67} to find $k\in\N$ such that
  \bas
	\int_{t_k-\tau}^{t_k} \|\Theta(\cdot,t)-\Tinf\|_{L^\infty(\Om)} dt \le \eta.
  \eas
  As $N$ is a null set, this ensures the existence of $\ts\in (t_k-\tau,t_k)\sm N$ such that
  $\|\Theta(\cdot,\ts)-\Tinf\|_{L^\infty(\Om)} \le \frac{\eta}{\tau}$,
  so that since $\ts\not\in\N$, we may draw on (\ref{51.87}) to fix $\epss\in (0,1)$ such that 
  \be{68.6}
	\|\Teps(\cdot,\ts)-\Tinf\|_{L^\infty(\Om)}
	\le \frac{2\eta}{\tau}
	\qquad \mbox{for all } \eps\in (0,\epss)\cap (\eps_j)_{j\in\N}.
  \ee
  But due to the first restriction in (\ref{68.5}), we may apply (\ref{68.3}) here to see that thus
  \bas
	\|\Teps(\cdot,t)-\Tinf\|_{L^\infty(\Om)}
	\le c_2 \|\Teps(\cdot,\ts)-\Tinf\|_{L^\infty(\Om)}
	+ c_2(t-\ts)^b
	\qquad \mbox{for all $t\in (\ts,\ts+\frac{1}{2})$,}
  \eas
  which since $t_k<\ts+\tau<\ts+\frac{1}{2}$ particularly warrants that as a consequence of (\ref{68.6}), the second inequality
  in (\ref{68.5}) and (\ref{68.4}),
  \bas
	\|\Teps(\cdot,t_k)-\Tinf\|_{L^\infty(\Om)}
	\le c_2 \cdot \frac{2\eta}{\tau} + c_2 (t_k-\ts)^b
	\le c_2 \cdot \frac{2\eta}{\tau} + c_2 \tau^b
	\le \frac{c_1}{4}+\frac{c_1}{4} 
	= \frac{c_1}{2}
  \eas
  for all $\eps\in (0,\epss)\cap (\eps_j)_{j\in\N}$.
  As also $t_k\not\in N$, again using (\ref{51.87}) we can take $\eps=\eps_j\searrow 0$ here to obtain that
  \bas
	\|\Theta(\cdot,t_k)-\Tinf\|_{L^\infty(\Om)} \le \frac{c_1}{2},
  \eas
  which indeed is incompatible with (\ref{68.2}).
\qed
\mysection{Decay of $u$. Proof of Theorem \ref{theo18}}
Our strategy toward asserting large-time extinction of $u$ will now go back to (\ref{0zeps}) and attempt to appropriately combine
(\ref{67.111}) with Lemma \ref{lem68} in order to show that the driving source $\vepsx$ appearing therein 
must asymptotically vanish in a suitable sense.\abs
Here, the ambition to conclude from (\ref{67.111}) that also the corresponding expressions $\zepsx$ 
undergo some large-time decay motivates the following lemma 
which utilizes appropriately regularized relatives of $\Theta$ as test functions in (\ref{wt}) to turn
the weak approximation property in (\ref{51.9}) into a statement on strong convergence:
\begin{lem}\label{lem65}
  Let $\tau>0$ and $T>\tau$, and let $(\Teps)_{\eps\in (0,1)}$ and $(\eps_j)_{j\in\N}$ be as in Lemma \ref{lem0} and
  Lemma \ref{lem51}. Then
  \be{65.1}
	\Tepsx \to \Theta_x
	\quad \mbox{in } L^2(\Om\times (\tau,T))
	\qquad \mbox{as } \eps=\eps_j\searrow 0.
  \ee
\end{lem}
\proof
  We once more rely on (\ref{51.87}) and moreover use that $\Theta\in L^2_{loc}(\bom\times (0,\infty))$ in choosing a null set
  $N_1 \subset (0,\infty)$ such that
  \be{65.01}
	\Teps(\cdot,t)\to \Theta(\cdot,t)
	\quad \mbox{in } L^2(\Om)
	\quad \mbox{for all } t\in (0,\infty)\sm N_1
	\qquad \mbox{as } \eps=\eps_j\searrow 0,
  \ee
  and that each $t\in (0,\infty)\sm N_1$ is a Lebesgue point of $0<t\mapsto \io \Theta^2(\cdot,t)$,
  and fixing any $\ts\in (0,\tau)\sm N_1$, $t_0\in (T,\infty)\sm N_1$ and $\del\in (0,\ts-\del)$ we let 
  \be{zd}
	\zd(t):=\lball
	0, & t\in (-\infty),\ts-\del), \\[1mm]
	\frac{t-\ts+\del}{\del}, \qquad 
	& t\in [\ts-\del,\ts), \\[1mm]
	1, & t\in [\ts,t_0), \\[1mm]
	\frac{\ts+\del-t}{\del}, \qquad 
	& t\in [t_0,t_0+\del), \\[1mm]
	0, & t\ge t_0+\del.
	\ear
  \ee
  For $h>0$, we furthermore write
  \be{Sh}
	[S_h \phi](x,t):=\frac{1}{h} \int_{t-h}^t \phi(x,s) ds,
	\qquad (x,t)\in \Om\times (0,\infty), \ \phi\in L^1_{loc}(\Om\times\R),
  \ee
  and note that, as can readily be verified (\cite{dibenedetto}), for each $T>0$ and any $\phi\in L^2(\Om\times\R)$
  it follows that
  \be{S1}
	S_h \phi \wto \phi
	\quad \mbox{in } L^2(\Om\times (0,T))
	\qquad \mbox{as } h\searrow 0,
  \ee
  while given any $\phi\in L^\infty(\Om\times \R)$, we have
  \be{S2}
	S_h \phi \wsto \phi
	\quad \mbox{in } L^\infty(\Om\times (0,T))
	\qquad \mbox{as } h\searrow 0.
  \ee
  Then for any $\del\in (0,\ts-\del)$, (\ref{zd}) ensures that 
  $\zd\in W^{1,\infty}(\R)$ with $\zd'\equiv 0$ on $(-\infty,\ts-\del)$ and on $(\ts,t_0)\cup (t_0+\del,\infty)$,  that
  $\zd'\equiv \frac{1}{\del}$ on $(\ts-\del,\ts)$, and that $\zd'\equiv -\frac{1}{\del}$ on $(t_0,t_0+\del)$,
  whence after trivially extending $\Theta$ by letting $\Theta(x,t):=0$ for $(x,t)\in\Om\times (-\infty,0)$, 
  in view of the inclusions $\Theta\in L^\infty_{loc}(\bom\times (0,\infty))$
  and $\Theta_x\in L^2_{loc}(\bom\times (0,\infty))$
  we infer from a standard approximation argument
  that we may use $\vp(x,t):=\zd(t) \cdot [S_h\Theta](x,t)$, $(x,t)\in \bom\times [0,\infty)$, $\del\in (0,\ts)$, 
  $h\in (0,\ts-\del)$, as a test function in (\ref{wt}).
  We thereby see that
  \bea{65.2}
	& & \hs{-10mm}
	- \frac{1}{\del} \int_{\ts-\del}^{\ts} \io \Theta S_h \Theta
	+ \frac{1}{\del} \int_{t_0}^{t_0+\del} \io \Theta S_h \Theta
	- \int_0^\infty \io \zd(t) \Theta(x,t) \cdot \frac{\Theta(x,t)-\Theta(x,t-h)}{h} dxdt \nn\\
	&=& - \int_0^\infty \io \zd(t) \Theta_x(x,t) [S_h \Theta_x](x,t) dxdt
	+ \int_0^\infty \io \zd(t) f'(\Theta(x,t)) \Theta_x(x,t) u_t(x,t) [S_h \Theta](x,t) dxdt \nn\\
	& & + \int_0^\infty \io \zd(t) f(\Theta(x,t)) u_t(x,t) [S_h \Theta_x](x,t) dxdt
	\qquad \mbox{for all $\del\in (0,\ts)$ and $h\in (0,\ts-\del)$},
  \eea
  where by Young's inequality and a linear substitution in the time variable,
  \bas
	& & \hs{-30mm}
	- \int_0^\infty \io \zd(t) \Theta(x,t) \cdot \frac{\Theta(x,t)-\Theta(x,t-h)}{h} dxdt \nn\\ 
	&\le& - \frac{1}{2h} \int_0^\infty \io \zd(t) \Theta^2(x,t) dxdt
	+ \frac{1}{2h} \int_0^\infty \io \zd(t) \Theta^2(x,t-h) dxdt \\
	&=& \frac{1}{2} \int_0^\infty \io \frac{\zd(t+h)-\zd(t)}{h} \cdot \Theta^2(x,t) dxdt \\
	&\to& \frac{1}{2} \int_0^\infty \io \zd'(t) \Theta^2(x,t) dxdt \\
	&=&  \frac{1}{2\del} \int_{\ts-\del}^{\ts} \io \Theta^2
	- \frac{1}{2\del} \int_{t_0}^{t_0+\del} \io \Theta^2
	\qquad \mbox{as } h\searrow 0.
  \eas
  Since $S_h \Theta \wsto \Theta$ in $L^\infty_{loc}(\bom\times (0,\infty))$ and $S_h \Theta_x \wto \Theta_x$ 
  in $L^2_{loc}(\bom\times (0,\infty))$ as $h\searrow 0$ by (\ref{S2}) and (\ref{S1}), 
  additionally using that $u_t\in L^2_{loc}(\bom\times (0,\infty))$
  we thus conclude on letting $h\searrow 0$ in (\ref{65.2}) that
  \bas
	\hs{-8mm}
	\int_0^\infty \io \zd(t) \Theta_x^2(x,t) dxdt
	&\ge& \frac{1}{2\del} \int_{\ts-\del}^{\ts} \io \Theta^2
	- \frac{1}{2\del} \int_{t_0}^{t_0+\del} \io \Theta^2 \nn\\
	& & + \int_0^\infty \io \zd(t) f'(\Theta(x,t)) \Theta(x,t) \Theta_x(x,t) u_t(x,t) dxdt \nn\\
	& & + \int_0^\infty \io \zd(t) f(\Theta(x,t)) \Theta_x(x,t) u_t(x,t) dxdt
	\quad \mbox{for all } \del\in (0,\ts),
  \eas
  which, according to (\ref{zd}) and the Lebesgue point properties of $\ts$ and $t_0$, in the limit $\del\searrow 0$ yields
  \be{65.3}
	\int_{\ts}^{t_0} \io \Theta_x^2
	\ge \frac{1}{2} \io \Theta^2(\cdot,\ts) - \frac{1}{2} \io \Theta^2(\cdot,t_0)
	+ \int_{\ts}^{t_0} \io f'(\Theta)\Theta \Theta_x u_t
	+ \int_{\ts}^{t_0} \io f(\Theta) \Theta_x u_t.
  \ee
  On the other hand, our choices of $\ts$ and $t_0$ moreover provide access to (\ref{65.01}), meaning that as $\eps=\eps_j\searrow 0$
  we have
  \bas
	\frac{1}{2} \io \Teps^2(\cdot,\ts)
	- \frac{1}{2} \io \Teps^2(\cdot,t_0)
	\to \frac{1}{2} \io \Theta^2(\cdot,\ts) - \frac{1}{2} \io \Theta^2(\cdot,t_0).
  \eas
  Apart from that, using (\ref{36.1}) together with (\ref{fe1})
  and the fact that $\ts>0$ we find $c_1>0$ and $\epss\in (0,1)$ such that
  \bas
	\feps'(\Teps) \Teps \le c_1
	\quad \mbox{and} \quad
	\feps(\Teps) \le c_1
	\quad \mbox{in } \Om\times (\ts,t_0)
	\qquad \mbox{for all } \eps\in (0,\epss),
  \eas
  which in combination with (\ref{fe2}) and the strong approximation feature in (\ref{51.7}) guarantees that thanks to
  the dominated convergence theorem
  \bas
	\feps'(\Teps) \Teps \veps
	\to f'(\Theta) \Theta u_t
	\quad \mbox{and} \quad
	\feps(\Teps) \veps
	\to f(\Theta) u_t
	\qquad \mbox{in } L^2(\Om\times (\ts,t_0))
  \eas
  as $\eps=\eps_j\searrow 0$.
  Therefore, the property in (\ref{51.9}) ensures that
  \bas
	& & \hs{-30mm}
	\frac{1}{2} \io \Teps^2(\cdot,\ts)
	- \frac{1}{2} \io \Teps^2(\cdot,t_0)
	+ \int_{\ts}^{t_0} \io \feps'(\Teps) \Teps \Tepsx \veps
	+ \int_{\ts}^{t_0} \io \feps(\Teps) \Tepsx \veps \nn\\
	&\to& \frac{1}{2} \io \Theta^2(\cdot,\ts) - \frac{1}{2} \io \Theta^2(\cdot,t_0)
	+ \int_{\ts}^{t_0} \io f'(\Theta) \Theta_x u_t
	+ \int_{\ts}^{t_0} \io f(\Theta) \Theta_x u_t
  \eas
  as $\eps=\eps_j\searrow 0$, and that again due to lower semicontinuity of $L^2$ norms with respect to weak convergence, also
  \bas
	\int_{\ts}^{t_0} \io \Theta_x^2 \le \liminf_{\eps=\eps_j\searrow 0} \int_{\ts}^{t_0} \io \Tepsx^2.
  \eas
  From (\ref{65.3}) it consequently follows that, in fact,
  \bas
	\int_{\ts}^{t_0} \io \Tepsx^2 \to \int_{\ts}^{t_0} \io \Theta_x^2
	\qquad \mbox{as } \eps=\eps_j\searrow 0,
  \eas
  and that thus, again by (\ref{51.9}), the claim results due to the inequalities $\ts<\tau<T<t_0$.
\qed
When utilizing the above in the intended direction, we will make use of
the following auxiliary statement, of which we include a brief proof for completeness.
\begin{lem}\label{lem42}
  There exists $C>0$ such that
  \be{42.1}
	\|\psi\|_{L^2(\Om)} \le C\|\psi_x\|_{(W^{1,2}(\Om))^\star}
	\qquad \mbox{for all } \psi\in C^1(\bom).
  \ee
\end{lem}
\proof
  For definiteness assuming that $\Om=(0,|\Om|)$ and $\|\chi_0\|_{W^{1,2}(\Om)}^2 = \io \chi_0^2 + \io \chi_{0x}^2$
  for $\chi_0\in W^{1,2}(\Om)$, given $\psi\in C^1(\bom)$ and $\chi\in C_0^\infty(\Om)$ we let $\chi_0(x):=\int_0^x \chi(x')dx'$,
  $x\in\bom$, and then obtain from the Cauchy-Schwarz inequality that $\chi_0^2 \le |\Om| \io \chi^2$ in $\Om$ and hence
  $\|\chi_0\|_{W^{1,2}(\Om)}^2 \le c_1\io \chi^2$ with $c_1:=|\Om|^2+1$. Therefore, an integration by parts shows that
  \bas
	\io \psi\chi
	= \io \psi \chi_{0x}
	= - \io \psi_x \chi_0
	\le \|\psi_x\|_{(W^{1,2}(\Om))^\star} \|\chi_0\|_{W^{1,2}(\Om)}
	\le \sqrt{c_1} \|\psi_x\|_{(W^{1,2}(\Om))^\star} \|\chi\|_{L^2(\Om)}
  \eas
  for any such $\chi$, so that
  \bas
	\|\psi\|_{L^2(\Om)}
	= \sup_{\begin{array}{cc} 
	\scriptstyle \chi\in C_0^\infty(\Om) \\[-1.5mm] 
	\scriptstyle 	\|\chi\|_{L^2(\Om)} \le 1 
	\end{array} } 
	\io \psi \chi
	\le \sqrt{c_1} \|\psi_x\|_{(W^{1,2}(\Om))^\star},
  \eas
  as claimed.
\qed
As announced, 
the second core step in our large time analysis now explicitly returns to (\ref{0zeps}) and the first two equations in (\ref{0eps}).
While the former will show that due to Lemma \ref{lem67} the averaged deviations 
$\int_{t_k}^{t_k+1} \|u(\cdot,t)-u(\cdot,t_k)\|_{L^2(\Om)}^2 dt$ must decay as $k\to\infty$ whenever $t_k\to\infty$,
the latter in conjunction with Lemma \ref{lem68} will identify all conceivable $\omega$-limits of $u$ as necessarily trivial.
In summary, this will yield the following.
\begin{lem}\label{lem43}
  If $u$ is as found in Lemma \ref{lem51}, then
  \be{43.1}
	u(\cdot,t) \to 0
	\quad \mbox{in } L^\infty(\Om)
	\qquad \mbox{as } t\to\infty.
  \ee
\end{lem}
\proof
  Since
  from the inclusion $u\in C^0(\bom\times [0,\infty)) \cap L^\infty((0,\infty);W^{1,2}(\Om))$ and the compactness of the embedding
  $W^{1,2}(\Om)\hra C^0(\bom)$ it readily follows that $(u(\cdot,t))_{t>0}$ is relatively compact in $C^0(\bom)$,
  and since the set $N$ found in Lemma \ref{lem0} satisfies $|N|=0$,
  it is sufficient to make sure that whenever $\uinf\in C^0(\bom)$ and $(t_k)_{k\in\N} \subset (0,\infty)\sm N$ are such that
  \be{43.2}
	u(\cdot,t_k) \to \uinf
	\quad \mbox{in } L^\infty(\Om)
  \ee
  as $k\to\infty$, we have $\uinf\equiv 0$.
  To this end, in a first step we combine (\ref{0eps}) with (\ref{0zeps}) to see, noting that both $u_{\eps xt}$ and
  $\uepsxxx$ are continuous in $\bom\times [0,\infty)$ for all $\eps\in (0,1)$ by parabolic regularity theory (\cite{LSU}), that
  with $(\zeps)_{\eps\in (0,1)}$ as in (\ref{zeps}),
  \bas
	\pa_t (\zeps-\uepsx)
	= \zepsxx - \feps'(\Teps) \zepsx^2 - \eps\uepsxxx
	\quad \mbox{in } \Om\times (0,\infty)
	\qquad \mbox{for all } \eps\in (0,1),
  \eas
  which implies that if we let $c_1>0$ be such that $\|\psi\|_{L^\infty(\Om)} + \|\psi_x\|_{L^2(\Om)} \le c_1$ for
  all $\psi\in C^1(\bom)$ fulfilling $\|\psi\|_{W^{1,2}(\Om)} \le 1$, then for any such $\psi$ we have
  \bas
	\bigg| \io \pa_t (\zeps-\uepsx) \psi\bigg|
	&=& \bigg| - \io \zepsx \psi_x
	- \io \feps'(\Teps) \zepsx^2 \psi
	+ \eps \io \uepsxx \psi_x \bigg| \\
	&\le& c_1 \|\zepsx\|_{L^2(\Om)} + 2c_1\Kf \|\zepsx\|_{L^2(\Om)}^2 + c_1 \eps\|\uepsxx\|_{L^2(\Om)}
	\qquad \mbox{for all $t>0$ and } \eps\in (0,1)
  \eas
  because of (\ref{fe1}).
  Therefore,
  \bas
	\big\| (\zeps-\uepsx)(\cdot,t)-(\zeps-\uepsx)(\cdot,t_k)\big\|_{(W^{1,2}(\Om))^\star} 
	&\le& c_1 \int_{t_k}^t \|\zepsx(\cdot,s)\|_{L^2(\Om)} ds
	+ 2c_1 \Kf \int_{t_k}^t \|\zepsx(\cdot,s)\|_{L^2(\Om)}^2 ds \\
	& & + c_1 \eps \int_{t_k}^t \|\uepsxx(\cdot,s)\|_{L^2(\Om)} ds
  \eas
  for all $t>t_k, k\in\N$ and $\eps\in (0,1)$,
  whence taking $c_2>0$ such that in line with Lemma \ref{lem42} we have 
  $\|\psi\|_{L^2(\Om)} \le c_2 \|\psi_x\|_{(W^{1,2}(\Om))^\star}$ for all $\psi\in C^1(\bom)$, we obtain that
  \bea{43.99}
	\|\ueps(\cdot,t)-\ueps(\cdot,t_k)\|_{L^2(\Om)}
	&\le& c_2\|\uepsx(\cdot,t)-\uepsx(\cdot,t_k)\|_{(W^{1,2}(\Om))^\star} \nn\\
	&\le& c_2 \big\| (\zeps-\uepsx)(\cdot,t)-(\zeps-\uepsx)(\cdot,t_k)\big\|_{(W^{1,2}(\Om))^\star} \nn\\
	& & + c_2 \| \zeps(\cdot,t)-\zeps(\cdot,t_k) \|_{(W^{1,2}(\Om))^\star} \nn\\
	&\le& c_1 c_2 \int_{t_k}^t \|\zepsx(\cdot,s)\|_{L^2(\Om)} ds
	+ 2c_1 c_2 \Kf \int_{t_k}^t \|\zepsx(\cdot,s)\|_{L^2(\Om)}^2 ds \nn\\
	& & + c_1 c_2 \eps \int_{t_k}^t \|\uepsxx(\cdot,s)\|_{L^2(\Om)} ds \nn\\
	& & + c_2 \| \zeps(\cdot,t)-\zeps(\cdot,t_k) \|_{(W^{1,2}(\Om))^\star} \nn\\
  \eea
  for each $t>t_k$, all $k\in\N$ and any $\eps\in (0,1)$.
  Here, the Cauchy-Schwarz inequality and (\ref{2.5}) guarantee that for all $t>t_k$ and $k\in\N$ we have
  \be{43.999}
	c_1 c_2 \eps \int_{t_k}^t \|\uepsxx(\cdot,s)\|_{L^2(\Om)} ds
	\le c_1 c_2 \eps^\frac{1}{2} \cdot \bigg\{ \eps \int_{t_k}^t \io \uepsxx^2 \bigg\}^\frac{1}{2} \cdot (t-t_k)^\frac{1}{2}
	\to 0
  \ee
  as $\eps\searrow 0$, and a combination of (\ref{51.8}) with the strong convergence property in Lemma \ref{lem65}
  and the definition in (\ref{z}) implies that 
  $\zepsx \to z_x$ in $L^2_{loc}(\bom\times (0,\infty))$ as $\eps=\eps_j\searrow 0$, while (\ref{51.87}) and (\ref{36.2}) warrant that
  for all $t\in (0,\infty)\sm N$ we have
  $\zeps(\cdot,t)\to z(\cdot,t)$ in $C^0(\bom)$ and hence also in $(W^{1,2}(\Om))^\star$ as $\eps=\eps_j\searrow 0$.
  In view of (\ref{51.5}) and the fact that $t_k\not\in N$, 
  on letting $\eps=\eps_j\searrow 0$ in (\ref{43.99}) we thus infer that for all $t\in (t_k,\infty)\sm N$ and $k\in\N$,
  \bas
	\|u(\cdot,t)-u(\cdot,t_k)\|_{L^2(\Om)}
	&\le& c_1 c_1 \int_{t_k}^t \|z_x(\cdot,s)\|_{L^2(\Om)} ds
	+ 2c_1 c_2 \Kf \int_{t_k}^t \|z_x(\cdot,s)\|_{L^2(\Om)}^2 ds \\
	& & + c_2 \| z(\cdot,t)-z(\cdot,t_k) \|_{(W^{1,2}(\Om))^\star}.
  \eas
  Again by the Cauchy-Schwarz inequality, the integrability property in (\ref{67.111}) together with 
  the stabilization feature in Lemma \ref{lem68} entails that therefore
  \bas
	\int_{t_k}^{t_k+1} \io \big( u(x,t)-u(x,t_k) \big)^2 dxdt \to 0
	\qquad \mbox{as } k\to\infty,
  \eas  
  and that thus, by (\ref{43.2}),
  \bas
	\int_{t_k}^{t_k+1} \io \big( u(x,t)-\uinf(x) \big)^2 dxdt \to 0
	\qquad \mbox{as } k\to\infty,
  \eas
  meaning that if we let $U_k(x,s):=u(x,t_k+s)$ for $x\in\Om, s\in (0,1)$ and $k\in\N$, then
  \be{43.3}
	U_k \to \uinf
	\qquad \mbox{in } L^2(\Om\times (0,1))
	\qquad \mbox{as } k\to\infty.
  \ee
  To confirm that this is possible only when $\uinf$ is trivial, we fix $\zeta\in C_0^\infty(\R)$ such that
  $\supp \zeta \subset (0,1)$ and $\int_0^1 \zeta=1$, and given $\chi\in C_0^\infty(\Om)$ we multiply the identity
  \bas
	\pa_t (\uepst - \eps \uepsxx) = - \eps \vepsxxxx+\uepsxx - \pa_x \feps(\Teps),
  \eas
  as resulting for $\eps\in (0,1)$ from (\ref{0eps}), by 
  $\bom\times [0,\infty) \ni (x,t) \mapsto \zeta(t-t_k) \chi(x)$, $k\in\N$, to obtain that
  \bea{43.9999}
	& & \hs{-20mm}
	\int_{t_k}^{t_k+1} \io \zeta''(t-t_k) \ueps(x,t) \chi(x) dxdt
	- \eps \int_{t_k}^{t_k+1} \io \zeta'(t-t_k) \uepsxx(x,t) \chi(x) dxdt \nn\\
	&=& - \eps \int_{t_k}^{t_k+1} \io \zeta(t-t_k) \vepsxx(x,t) \chi_{xx}(x) dxdt
	+ \int_{t_k}^{t_k+1} \io \zeta(t-t_k) \ueps(x,t) \chi_{xx}(x) dxdt \nn\\
	& & + \int_{t_k}^{t_k+1} \io \zeta(t-t_k) \feps(\Teps(x,t)) \chi_x(x) dxdt
  \eea
  for all $k\in\N$ and $\eps\in (0,1)$.
  Here, a reasoning similar to that in (\ref{43.999}) shows that due to the Cauchy-Schwarz 
  inequality, (\ref{2.5}) and (\ref{2.4}),
  \bas
	\bigg| 	- \eps \int_{t_k}^{t_k+1} \io \zeta'(t-t_k) \uepsxx(x,t) \chi(x) dxdt \bigg|
	&\le& \|\zeta'\|_{L^\infty(\R)} \|\chi\|_{L^2(\Om)} \cdot \eps \int_{t_k}^{t_k+1} \|\uepsxx(\cdot,t)\|_{L^2(\Om)} dt \\
	&\le& \|\zeta'\|_{L^\infty(\R)} \|\chi\|_{L^2(\Om)} \cdot \sqrt{\eps} \cdot 
		\bigg\{ \eps \int_{t_k}^{t_k+1} \io \uepsxx^2 \bigg\}^\frac{1}{2} \\[2mm]
	&\to& 0
	\qquad \mbox{as } \eps\searrow 0
  \eas
  and
  \bas
	\bigg| - \eps \int_{t_k}^{t_k+1} \io \zeta(t-t_k) \vepsxx(x,t) \chi_{xx}(x) dxdt \bigg|
	&\le& \|\zeta\|_{L^\infty(\R)} \|\chi_{xx}\|_{L^2(\R)} \cdot \eps \int_{t_k}^{t_k+1} \|\vepsxx(\cdot,t)\|_{L^2(\Om)} dt \\
	&\le& \|\zeta\|_{L^\infty(\R)} \|\chi_{xx}\|_{L^2(\R)} \cdot \sqrt{\eps} \cdot
		\bigg\{ \eps \int_{t_k}^{t_k+1} \io \vepsxx^2 \bigg\}^\frac{1}{2} \\[2mm]
	&\to& 0
	\qquad \mbox{as } \eps\searrow 0,
  \eas
  whereas (\ref{51.8}) together with (\ref{36.1}) and (\ref{fe2}) implies that
  \bas
	\int_{t_k}^{t_k+1} \io \zeta(t-t_k) \feps(\Teps(x,t)) \chi_x(x) dxdt
	\to \int_{t_k}^{t_k+1} \io \zeta(t-t_k) f(\Theta(x,t)) \chi_x(x) dxdt
	\qquad \mbox{as } \eps=\eps_j\searrow 0
  \eas
  thanks to the dominated convergence theorem.
  Since from (\ref{51.5}) we furthermore know that as $\eps=\eps_j\searrow 0$ we have
  \bas
	\int_{t_k}^{t_k+1} \io \zeta''(t-t_k) \ueps(x,t) \chi(x) dxdt
	\to 
	\int_{t_k}^{t_k+1} \io \zeta''(t-t_k) u(x,t) \chi(x) dxdt
  \eas
  and
  \bas
	\int_{t_k}^{t_k+1} \io \zeta(t-t_k) \ueps(x,t) \chi_{xx}(x) dxdt 
	\to \int_{t_k}^{t_k+1} \io \zeta(t-t_k) u(x,t) \chi_{xx}(x) dxdt,
  \eas
  letting $\eps=\eps_j\searrow 0$ in (\ref{43.9999}) thus leads to the identity
  \bas
	\int_{t_k}^{t_k+1} \io \zeta''(t-t_k) u(x,t) \chi(x) dxdt
	&=& \int_{t_k}^{t_k+1} \io \zeta(t-t_k) u(x,t) \chi_{xx}(x) dxdt \\
	& & + \int_{t_k}^{t_k+1} \io \zeta(t-t_k) f(\Theta(x,t)) \chi_x(x) dxdt
	\qquad \mbox{for all } k\in\N,
  \eas
  which by definition of the $U_k$ is equivalent to
  \bea{43.4}
	\hs{-8mm}
	\int_0^1 \io \zeta''(s) U_k(x,s) \chi(x) dxds
	&=& \int_0^1 \io \zeta(s) U_k(x,s) \chi_{xx}(x) dxds \nn\\
	& & + \int_0^1 \io \zeta(s) f(\Theta(x,t_k+s)) \chi_x(x) dxds
	\quad \mbox{for all } k\in\N.
  \eea
  But from (\ref{43.3}) we know that here
  \bas
	\int_0^1 \io \zeta''(s) U_k(x,s) \chi(x) dxds
	\to \int_0^1 \io \zeta''(s) \uinf(x) \chi(x) dxds
	= 0
	\qquad \mbox{as } k\to\infty,
  \eas
  and that
  \bas
	\int_0^1 \io \zeta(s) U_k(x,s) \chi_{xx}(x,s) dxds
	&\to& \int_0^1 \io \zeta(s) \uinf(x) \chi_{xx}(x) dxds \\
	&=& \io \uinf(x) \chi_{xx}(x) dx
	\qquad \mbox{as } k\to\infty,
  \eas
  because $\int_0^1 \zeta''(s) ds=\zeta'(1)-\zeta'(0)=0$ and $\int_0^1 \zeta(s) ds=1$.
  Similarly, Lemma \ref{lem68} and the continuity of $f$ ensure that since $\io \chi_x=0$,
  \bas
	\int_0^1 \io \zeta(s) f(\Theta(x,t_k+s)) \chi_x(x) dxds
	&\to& \int_0^1 \io \zeta(s) f(\Tinf) \chi_x(x) dxds \\
	&=& \io f(\Tinf) \chi_x(x) dx \\[2mm]
	&=& 0
	\qquad \mbox{as } k\to\infty.
  \eas
  Consequently, (\ref{43.4}) entails that $\io \uinf \chi_{xx}=0$ for all $\chi\in C_0^\infty(\Om)$, which by density implies that
  \bas
	\io \uinf \chi_{xx}=0
	\qquad \mbox{for all $\chi\in C^2(\bom)$ fulfilling $\chi=0$ on } \pO.
  \eas
  As $\uinf$ is continuous, we may here choose $\chi$ to be the classical solution of $\chi_{xx}=\uinf$ in $\Om$ with
  $\chi|_{\pO}$ to infer that, indeed, $\uinf\equiv 0$.
\qed
It remains to collect tesserae:\abs
\proofc of Theorem \ref{theo18}.\quad
  We only need to combine Lemma \ref{lem43} with Lemma \ref{lem68}.
\qed

\bigskip

{\bf Acknowlegements.} \quad
The author acknowledges support of the Deutsche Forschungsgemeinschaft (Project No. 444955436).\abs
{\bf Conflict of interest statement.} \quad
The author declares that he has no conflict of interest, 
and that he has no relevant financial or non-financial interests to disclose.\abs
{\bf Data availability statement.} \quad
Data sharing is not applicable to this article as no datasets were
generated or analyzed during the current study.\abs

\end{document}